%% file: bbmh__arXiv.tex
\documentclass[]{scrartcl}

\usepackage[utf8]{luainputenc}
\usepackage[USenglish]{babel}
\usepackage{csquotes}

\usepackage[a4paper,top=27mm,bottom=20mm,inner=25mm,outer=20mm]{geometry}

\usepackage[%
  backend=bibtex,bibencoding=ascii,
  style=numeric-comp,
  giveninits=true, uniquename=init, 
  natbib=true,
  url=true,
  doi=true,
  isbn=false,
  backref=false,
  maxnames=99,
  ]{biblatex}
\usepackage{amsmath}
\allowdisplaybreaks
\numberwithin{equation}{section}
\usepackage{amssymb}
\usepackage{commath}
\usepackage{mathtools}
\usepackage{bbm}
\usepackage{nicefrac}
\usepackage{subdepth}

\usepackage{siunitx}
\usepackage{algorithm}
\usepackage{algpseudocode}

\usepackage{amsthm}
\usepackage{thmtools}
\usepackage{etoolbox}
\makeatletter
\patchcmd{\thmt@setheadstyle}
 {\bgroup\thmt@space}
 {\thmt@space}
 {}{}
\patchcmd{\thmt@setheadstyle}
 {\egroup\fi}
 {\fi}
 {}{}
\makeatother
\declaretheoremstyle[
  bodyfont=\normalfont\itshape,
  headformat=\NAME\ \NUMBER\NOTE,
]{myplain}
\declaretheoremstyle[
  headformat=\NAME\ \NUMBER\NOTE,
]{mydefinition}
\newcommand{\envqed}{{\lower-0.3ex\hbox{$\triangleleft$}}}
\declaretheorem[style=myplain,numberwithin=section]{theorem}

\declaretheorem[style=mydefinition,numberlike=theorem,qed=\envqed]{definition}
\declaretheorem[style=mydefinition,numberlike=theorem,qed=\envqed]{remark}

\usepackage[plainpages=false,pdfpagelabels,hidelinks,unicode]{hyperref}

\usepackage{color}
\usepackage{graphicx}
\usepackage[small]{caption}
\usepackage{subcaption}

\begingroup\expandafter\expandafter\expandafter\endgroup
\expandafter\ifx\csname pdfsuppresswarningpagegroup\endcsname\relax
\else
\fi

\usepackage{booktabs}
\usepackage{rotating}
\usepackage{multirow}

\usepackage{enumitem}

\usepackage{ifluatex}
\ifluatex
  \usepackage[no-math]{fontspec}
\else
  \usepackage[T1]{fontenc}
\fi
\usepackage{newpxtext,newpxmath}

\input{macros}

\newcommand{\orcid}[1]{ORCID:~\href{https://orcid.org/#1}{#1}}
\usepackage{authblk}

\newenvironment{keywords}{\par\textbf{Key words.}}{\par}
\newenvironment{AMS}{\par\textbf{AMS subject classification.}}{\par}

\newcommand{\eps}{\varepsilon}

\title{Asymptotic-preserving and energy-conserving methods for a hyperbolic approximation of the BBM equation}

\author[1]{Sebastian~Bleecke}
\affil[1]{Institute of Mathematics, Johannes Gutenberg University Mainz, Staudingerweg 9, 55128 Mainz, Germany}

\author[2]{Abhijit Biswas\thanks{\orcid{0000-0002-0741-6303}}}
\affil[2]{Department of Mathematics and Statistics, Indian Institute of Technology Kanpur, Kanpur, India}

\author[3]{David~I.~Ketcheson\thanks{\orcid{0000-0002-1212-126X}}}
\affil[3]{King Abdullah University of Science and Technology (KAUST),
Computer Electrical and Mathematical Science and Engineering Division (CEMSE),
Thuwal, 23955-6900, Saudi Arabia}

\author[1]{Hendrik~Ranocha\thanks{\orcid{0000-0002-3456-2277}}}

\author[4]{Jochen Schütz\thanks{\orcid{0000-0002-6355-9130}}}
\affil[4]{Faculty of Sciences \& Data Science Institute, Hasselt University, Agoralaan Gebouw D, BE-3590 Diepenbeek, Belgium}

\date{November 11, 2025} 

\makeatletter
\hypersetup{pdfauthor={Sebastian~Bleecke, Abhijit Biswas, David~I. Ketcheson, Hendrik~Ranocha, Jochen~Schütz}} 
\hypersetup{pdftitle={Asymptotic-preserving and energy-conserving methods for a hyperbolic approximation of the BBM equation}} 
\makeatother

\begin{document}

\maketitle

\begin{abstract}
\noindent
  \input{abstract.tex}
\end{abstract}



\section{Introduction}
The Benjamin-Bona-Mahony (BBM) equation~\cite{BBM1972}
\begin{equation}
\label{eq:BBM}
\uBBM_t + \uBBM \uBBM_x - \uBBM_{txx} = 0,
\end{equation}
originally introduced by Peregrine \cite{peregrine1966} and also known as the
regularized long wave equation, is a model for weakly nonlinear surface waves in shallow water,
and a prototype for many dispersive wave models\footnote{This equation was
originally written with an additional linear term $+ \uBBM_x$, which can be removed
by the transformation $\uBBM \mapsto \uBBM - 1$.}.

Numerical solution of the BBM equation \eqref{eq:BBM} can be challenging due to the
presence of both the nonlinear convective term and the mixed third-derivative term.
For example, it can be rewritten as
\begin{equation}
  (\I - \partial_x^2) \uBBM_t + \uBBM \uBBM_x = 0
\end{equation}
and solved by integrating the nonlinear term using traditional methods, while
performing an elliptic solve at each time step.

A first-order hyperbolic approximation for \eqref{eq:BBM}, referred to as BBMH,
was introduced in \cite{gavrilyuk2022hyperbolic}.  It comprises
a system of three equations involving a relaxation parameter
$\epsilon$, and the system is formally equivalent to the BBM equation when $\epsilon \to 0$.
Numerical solution of this hyperbolization may be an appealing alternative
due to the availability of standard numerical methods for hyperbolic problems,
the relative ease of imposing boundary conditions, and avoiding the need to perform an
elliptic solve.  Furthermore, the BBMH system has a Hamiltonian structure that
is closely related to that of the original BBM equation.
On the other hand, when $\epsilon$ is small the hyperbolic system
involves fast waves that may pose their own numerical challenge.

Similar hyperbolic approximations have been proposed for many other higher
order partial differential equations (PDEs); see e.g., \cite{favrie2017rapid,busto2021high,guermond2022hyperbolic,ketcheson2025approximation,giesselmann2025convergence}.
Improvements in our understanding of how to efficiently and accurately solve the resulting
hyperbolic systems is of high interest as it may benefit a wide range of applications.

In the present work we propose and analyze numerical methods for the solution
of the BBMH system, proposing a class of methods that are energy conserving, asymptotic preserving,
and computationally efficient.  We employ
summation-by-parts (SBP) operators in space and implicit-explicit
Runge-Kutta (IMEX RK) methods in time, with time relaxation to achieve fully-discrete
energy conservation.
We numerically demonstrate that the resulting schemes exhibit improved long-term error growth
compared to non-conservative methods, as expected based on the relative equilibrium structure of the equations \cite{duran1998numerical,araujo2001error}.

While we focus on finite differences \cite{kreiss1974finite,strand1994summation,carpenter1994time} in space, we formulate our methods
in terms of general SBP operators. Thus, the same structure-preserving properties
also hold for Fourier collocation methods,
finite volumes \cite{nordstrom2001finite},
continuous Galerkin~(CG) finite elements \cite{hicken2016multidimensional,hicken2020entropy,abgrall2020analysisI},
discontinuous Galerkin~(DG) methods \cite{gassner2013skew,carpenter2014entropy},
flux reconstruction~(FR) \cite{huynh2007flux,vincent2011newclass,ranocha2016summation},
active flux methods \cite{eymann2011active,barsukow2025stability},
and some meshless schemes \cite{hicken2024constructing}.
See \cite{svard2014review,fernandez2014review} for reviews of SBP methods.

The relaxation-in-time methods we use to obtain fully-discrete energy conservation
are small modifications of standard IMEX methods. The origin of these methods
dates back to \cite{sanzserna1982explicit} and \cite[pp.~265--266]{dekker1984stability}.
More recently, they have been developed in \cite{ketcheson2019relaxation,ranocha2020relaxation,ranocha2020general}.
Their combination with IMEX methods has been considered in \cite{kang2022entropy,li2022implicit}.
Besides Hamiltonian problems \cite{ranocha2020relaxationHamiltonian,zhang2020highly,li2023relaxation},
some applications include
compressible flows \cite{yan2020entropy,ranocha2020fully,doehring2025paired}
and dispersive wave equations \cite{ranocha2025structure,lampert2024structure,mitsotakis2021conservative}.

The rest of the paper is organized as follows.
In Section~\ref{sec:rel_eq} we introduce a hyperbolic approximation of the BBM
equation and show that it possesses a relative equilibrium structure
analogous to that of the BBM equation itself.
In Section \ref{sec:traveling_waves} we discuss traveling wave solutions of BBMH
and describe the generation of numerical solitary wave solutions via the Petviashvili method.
We show that appropriately chosen
IMEX RK discretizations of BBMH are indeed asymptotic preserving in Section~\ref{sec:continuous_in_space}.
In Section~\ref{sec:discrete_in_space} we develop a spatial discretization for BBMH
that exactly conserves the modified energy on a semi-discrete level, via the use of
SBP operators.
Finally, in Section~\ref{sec:numerical_experiments} we perform numerical tests
of the fully-discrete scheme, confirming our theoretical results and demonstrating
its efficiency.

\section{Relative equilibrium structure}\label{sec:rel_eq}

The BBM equation considered on the interval $[x_\mathrm{min}, x_\mathrm{max}]$ can be written as a Hamiltonian system \cite{Olver_1980}
\begin{equation}
  \uBBM_t = \JBBM \delta \HBBM(\uBBM),
\end{equation}
where $\JBBM$ is a skew-symmetric operator
\begin{equation}
  \JBBM = -(\I-\partial_x^2)^{-1} \partial_x,
\end{equation}
$\delta$ represents the variational derivative, and $\HBBM(\uBBM)$ is the Hamiltonian
\begin{equation}
  \HBBM(\uBBM) = \int_{x_\mathrm{min}}^{x_\mathrm{max}}\frac{\uBBM^3}{6} \dif x.
\end{equation}
The BBM equation possesses an additional conserved quantity
\begin{equation}\label{eq:I_BBM}
  \IBBM(\uBBM) = \int_{x_\mathrm{min}}^{x_\mathrm{max}}\left(\frac{\uBBM^2}{2} + \frac{\uBBM_x^2}{2}\right)  \dif x
\end{equation}
which satisfies
\begin{equation}
  \uBBM_x = \JBBM \delta \IBBM(\uBBM).
\end{equation}
For numerical methods conserving $\HBBM$ or $\IBBM$ for solitary wave solutions, this relative equilibrium structure leads to a reduced long-time error growth $O(t)$ instead of $O(t^2)$ \cite{duran1998numerical,araujo2001error}.

The BBM equation can also be written in terms of the Lagrangian \cite{gavrilyuk2022hyperbolic}
\begin{equation}
  \LBBM(\varphi) = -\frac{\varphi_t(\varphi_x - 1)}{2} + \frac{\varphi_t \varphi_{xxx}}{2} - \frac{\varphi_x^3 - 1}{6}.
\end{equation}
The Euler-Lagrange equation for $\LBBM$ is
\begin{equation}
  \begin{aligned}
    &\varphi_{xt} + \varphi_x \varphi_{xx} - \varphi_{xtxx} = 0.
  \end{aligned}
\end{equation}
Setting $\uBBM = \varphi_x$, this is equivalent to the BBM equation \eqref{eq:BBM}.
Based on this, Gavrilyuk and Shyue \cite{gavrilyuk2022hyperbolic}
introduced the modified Lagrangian
\begin{equation}
  \LBBMH(\phi,\psi, \chi) = -\frac{\phi_t \phi_x}{2} - \frac{\phi_x^3}{6} - \varepsilon^2 \frac{\psi_t \psi_x}{2} - \phi_x \psi_x + \psi_x \chi - \frac{\chi_t \chi_x}{2} - \varepsilon^2\frac{\chi_x^2}{2},
\end{equation}
for which the Euler-Lagrange equations read
\begin{equation}\label{eq:BBMH}
  \begin{aligned}
    u_t + u u_x + v_x &= 0,\\
    v_t + \frac{1}{\varepsilon^2} u_x &= \frac{1}{\varepsilon^2} w,\\
    w_t + \varepsilon^2 w_x &= -v,
  \end{aligned}
\end{equation}
where $u = \phi_x, v=\psi_x$, and $w=\chi_x$.
As shown in \cite{gavrilyuk2022hyperbolic}, \eqref{eq:BBMH} can be written
(for any $\epsilon>0$) as a symmetric hyperbolic system\footnote{Gavrilyuk and Shyue \cite{gavrilyuk2022hyperbolic}
introduced two parameters $c$ and $\widehat{c}$ in their Lagrangian such that the original
BBM equation is recovered for $c, \widehat{c} \to \infty$. We use
$c = \widehat{c} = 1/\varepsilon^2$ to simplify the analysis and for consistency with
the literature on asymptotic-preserving methods.}.
Formally, as $\epsilon \to 0$, \eqref{eq:BBMH} is equivalent to \eqref{eq:BBM},
and thus \eqref{eq:BBMH} can be viewed as a hyperbolic approximation of \eqref{eq:BBM}.
A rigorous convergence proof of the hyperbolic approximation to a smooth solution of the BBM equation is provided in \cite{giesselmann2025convergence}.
We refer to \eqref{eq:BBMH} as BBMH.
To write the BBMH system more compactly we will sometimes use the notation
\begin{equation}
  q = (u, v, w)^T.
\end{equation}
In Section \ref{sec:numerical_experiments} we will study the long-time behavior of the error in the
numerical solution of BBMH.  The nature of this behavior depends critically on
preservation of the \emph{relative equilibrium structure} of the equation
\cite{duran1998numerical}.
The relative equilibrium structure of the original BBM equation has been
described in \cite{araujo2001error}.

Since it is obtained from a Lagrangian, the BBMH system possesses a Hamiltonian structure.
Conservation laws for quantities analogous to those above are given in
\cite[Section~2, equations (14)--(15)]{gavrilyuk2022hyperbolic}.  The following
theorem provides the full relative equilibrium structure related to these
quantities.
\begin{theorem}\label{thm:relative_equilibrium_structure}
    The BBMH system~\eqref{eq:BBMH} is a Hamiltonian PDE
    $\partial_t q = \JBBMH \delta \HBBMH(q)$
    with relative equilibrium structure
    $\partial_x q = \JBBMH \delta \IBBMH(q)$,
    where
    \begin{equation}\label{eq:BBMH_total}
    \begin{gathered}
      \JBBMH= -\operatorname{diag}\left(\partial_x, \frac{1}{\eps^2}\partial_x, \partial_x\right),
      \quad
      \HBBMH(q) = \int_{x_\mathrm{min}}^{x_\mathrm{max}} \left(
        \frac{u^3}{6} + u v - v \chi + \epsilon^2 \frac{w^2}{2}
      \right) \dif x,
      \\
      \IBBMH(q) = \int_{x_\mathrm{min}}^{x_\mathrm{max}} \left(
        \frac{u^2}{2} + \epsilon^2 \frac{v^2}{2} + \frac{w^2}{2}
      \right) \dif x,
    \end{gathered}
    \end{equation}
    where we recall that $\chi_x = w$.
\end{theorem}
Note that formally, $\IBBMH(q) \to \IBBM(\uBBM)$ as $\epsilon \to 0$.

\begin{remark}
\label{rem:other_hyperbolic_approximation}
It is interesting to note that there exists another hyperbolic approximation
of the BBM equation, which amounts to
\begin{equation}
  \begin{aligned}
    u_t + uu_x - \frac{1}{\eps^2} (u_x - v) &= 0\\
    v_t + \frac{1}{\eps^2} (v_x - w) = 0\\
    w_t - \frac{1}{\eps^2} (u_x - v) = 0.
  \end{aligned}
\end{equation}
This choice of hyperbolization of the BBM equation also is linearly stable.
However, this system does not possess a Hamiltonian or relative equilibrium
structure, so in the present work we focus exclusively on \eqref{eq:BBMH}.
\end{remark}

\section{Traveling wave solutions}\label{sec:traveling_waves}
An analysis of some traveling wave solutions of the BBMH system has been performed in
\cite{gavrilyuk2022hyperbolic}, including both solitary waves and periodic waves.  As noted
there, there exist solitary wave solutions of BBMH that converge to solitary wave solutions
of BBM as $\epsilon \to 0$.  Taking a slightly different approach, we observe that the
traveling wave ansatz
\begin{align} \label{tw-ansatz}
    \bigl(u(x,t), v(x,t), w(x,t)\bigr) & = \bigl(\ut(x-ct), \vt(x-ct), \wt(x-ct)\bigr)
\end{align}
leads to a system of ODEs that can be reduced to the following pair of equations:
\begin{equation} \label{bbmh-traveling-odes}
    \ut' = \frac{\wt}{1 + c\epsilon^2(\ut-c)}, \qquad
    \wt' = \frac{\ut(\ut/2 - c)}{\epsilon^2-c}.
\end{equation}
Here we have assumed that $\ut$ and $\vt$ tend to zero asymptotically in
order to determine the constant of integration.
This system is very similar to that obtained previously for the KdVH
system \cite[eq.~(6)]{biswas2025traveling}, with the main difference being
that the denominator of the equation for $\wt'$ can vanish or become negative.
Analysis of the equilibria of this system shows that solitary waves with
vanishing background state exist if and only if $c>\epsilon^2$.  It is peculiar that,
whereas KdVH solitary waves are subject to a maximum speed, BBMH solitary
waves are subject to a minimum speed.  Note that as $\epsilon\to 0$, this
restriction becomes trivial, so there exist BBMH traveling waves that approximate
BBM solitary waves of any speed.

For $c < \epsilon^2$, there exist a variety of traveling wave solutions, including
both solitary and periodic waves, which do not correspond to any solution of the BBM equation.
These include analogs of all the types of waves shown in
\cite[Fig.~3]{biswas2025traveling}, as well as a family of peaked solitary waves, of
which an example is depicted in Figure \ref{fig:peakon}.

\begin{figure}[htbp]
  \centering
  \includegraphics{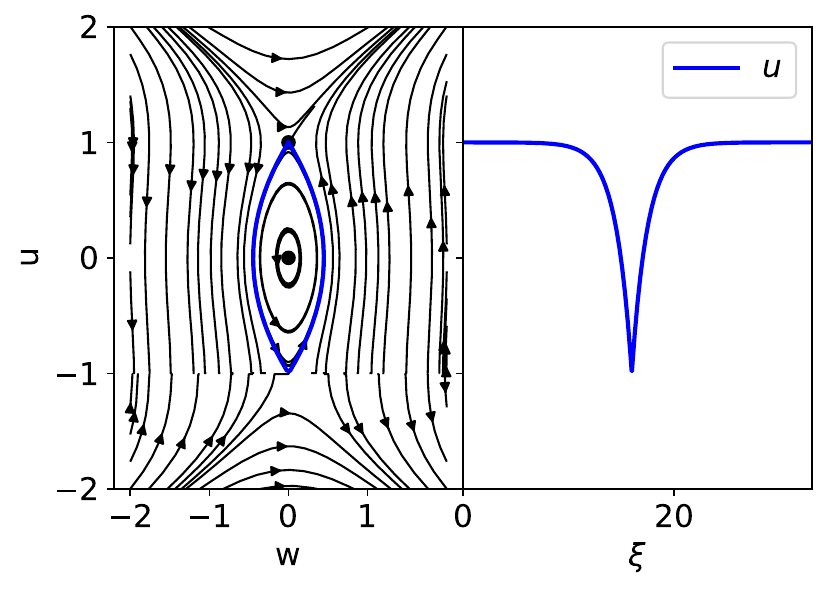}
  \caption{A peakon solution of the BBMH system, obtained by integrating \eqref{bbmh-traveling-odes} with $\epsilon^2=4/3$, $c=1/2$.}
  \label{fig:peakon}
\end{figure}

\subsection{Computation of BBMH solitary waves via Petviashvili's method}\label{sec:petv_derivation}
For some numerical tests in this work, we use numerically computed solitary
wave solutions of the BBMH system.  For completeness, we describe the calculation
of such solutions in this section.  For a more detailed analysis of BBMH solitary waves
and their convergence to BBM solitary waves, see \cite{gavrilyuk2022hyperbolic}.

Here we use the Petviashvili algorithm to construct traveling wave solutions numerically;
see e.g. \cite{ALVAREZ201439}.
It is convenient\footnote{In order to set the constants of integration to zero in what follows.}
to make the substitution $u\to u+1$ in \eqref{eq:BBMH} and work with the resulting
system
\begin{equation}\label{eq:petv_system}
    \begin{aligned}
      u_t + u_x + u u_x + v_x &= 0, \\
      v_t + \frac{1}{\epsilon^2} u_x &= \frac{1}{\epsilon^2} w, \\
      w_t + \epsilon^2 w_x &= -v.
    \end{aligned}
\end{equation}
We again apply the traveling wave ansatz \eqref{tw-ansatz}.
Integrating the first resulting equation, and using the remaining two
to eliminate $\vt, \wt$, we find
\begin{equation} \label{LNform}
 \left((c - 1) I - (I + (c - \eps^2) c \eps^2 \partial_{\xi}^2)^{-1} (c - \eps^2) \partial_{\xi}^2\right) \ut = \frac{\ut^2}{2},
\end{equation}
where $\xi = x-ct$.  Petviashvili's method consists of the iteration
\begin{equation} \label{eq:petv_fpi}
  L \ut^{[n + 1]} = m(\ut^{[n]})^\gamma N(\ut^{[n]}),
\end{equation}
where $L\ut$ and $N(\ut)$ are the left and right sides of \eqref{LNform}; i.e.
\begin{equation}
    L := \left((c - 1) I - (I + (c - \eps^2) c \eps^2 \partial_{\xi}^2)^{-1} (c - \eps^2) \partial_{\xi}^2\right),
    \qquad
    N(\ut) := \frac{\ut^2}{2},
\end{equation}
and
\begin{equation}
    m(\ut) = \frac{\langle L \ut, \ut\rangle}{\langle N \ut, \ut\rangle}.
\end{equation}
The value $\gamma$ should be chosen in order to improve the convergence of the method.
Following \cite{Pelinovsky2004} and \cite{ALVAREZ201439} we choose $\gamma = 2$.
From the solution $\ut$, we can obtain $v$ and $w$ from the traveling wave relations
\begin{equation}
    \vt = \left( c- 1 -\frac{1}{2}\ut \right)\ut,
    \qquad
    \wt = \ut' - c\epsilon^2 \vt'.
\end{equation}
The computed traveling wave solution $(\ut, \vt, \wt)$ of equation~\eqref{eq:petv_system} can be related to
the traveling wave solution of the original BBMH system~\eqref{eq:BBMH} by setting
\begin{equation}
  u = \ut - 1, \quad v = \vt, \quad w = \wt.
\end{equation}
To verify the numerically computed traveling wave solutions, we compare the resulting function $u$ to an
exact solitary wave solution of the BBM equation, which we also use as the initial guess $u^{[0]}$.
We discretize the derivatives using a Fourier collocation method on a grid with $N = 2^{10}$ points.
The result of the comparison, for $\eps = 10^{-1}, 10^{-3}$, is shown in Figure~\ref{fig:petv_validation_u}.

\begin{figure}[htbp]
  \centering
  \includegraphics[scale = 0.5]{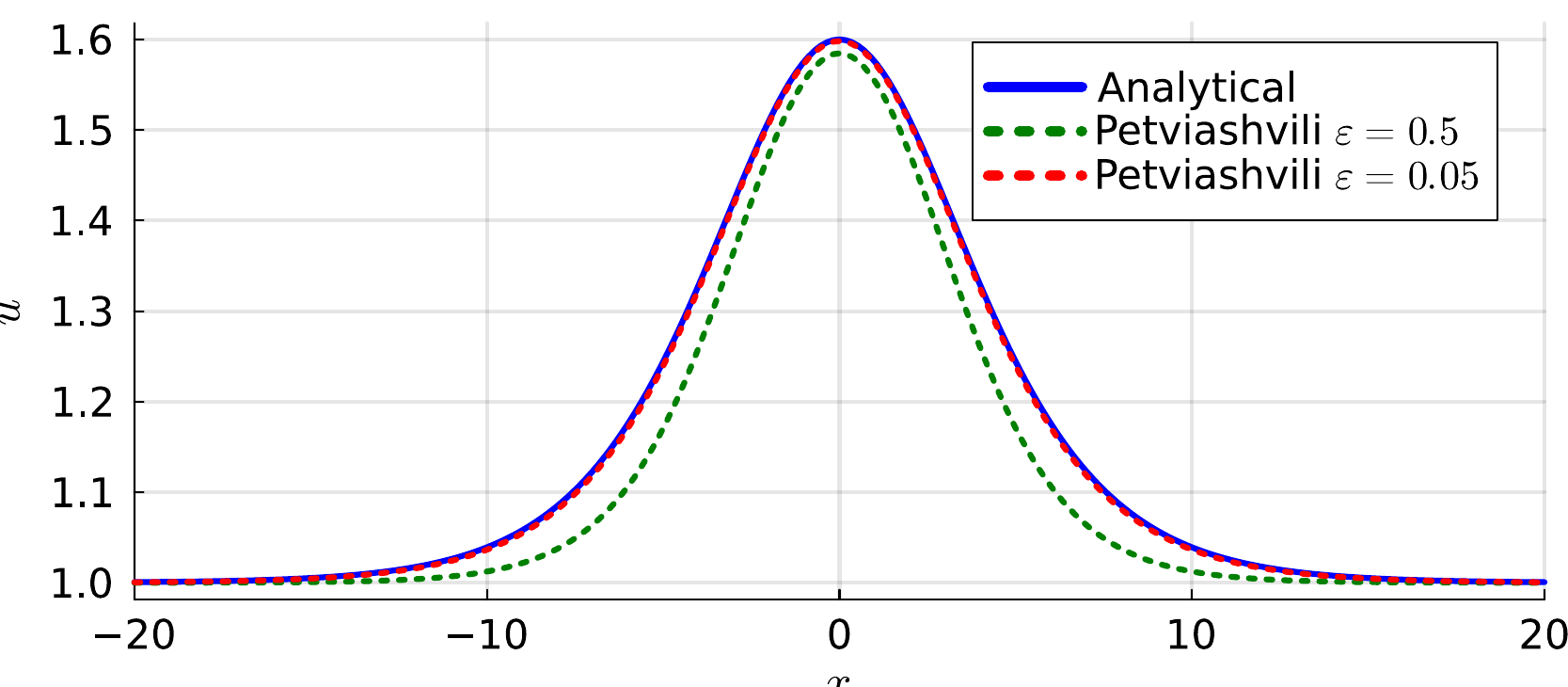}
  \caption{Numerical solution of the $u$ component of the BBMH system compared to the $u$ component of the analytical solution of the BBM equation.}
  \label{fig:petv_validation_u}
\end{figure}

\section{Splitting the BBMH system for IMEX methods: asymptotic-preserving time discretizations}\label{sec:continuous_in_space}
In this section we investigate time discretizations of the BBMH system.
\begin{definition}
  We call a numerical time discrete method applied to the BBMH system~\eqref{eq:BBMH} asymptotic preserving (AP) in $u$
  if the numerical update $u^{n + 1}$ converges towards the numerical update of the BBM solution $\uBBM^{n + 1}$, i.e.
  \begin{equation}
    u^{n + 1} \rightarrow \uBBM^{n + 1}
  \end{equation}
  for $\eps \rightarrow 0$. Similarly, we define the AP property for the $v$ and the $w$ component with the respective limits
  \begin{equation}
      v^{n + 1} \rightarrow -D_t \uBBM_x, \qquad
      w^{n + 1} \rightarrow \uBBM_x,
  \end{equation}
  where $D_t$ is a numerical time derivative induced by the applied numerical method.
\end{definition}
For a more detailed definition and discussion of the AP property see~\cite{biswas2025traveling} or~\cite{boscarino2024asymptotic}.
\subsection{General IMEX methods of type I and type~II}
From the BBMH system~\eqref{eq:BBMH} we observe that the evolution equation for $v$ becomes stiff as $\epsilon\to0$,
suggesting the use of an implicit time discretization.  In order to avoid implicit handling of the nonlinear
convective term, which would require the solution of large nonlinear algebraic systems, we explore the use of
IMEX Runge-Kutta methods, in which some terms are handled implicitly and others explicitly.
In this work we only consider diagonally implicit methods, which take the form
\begin{equation}
  \label{eq:imex_def}
  \begin{aligned}
    q^{(i)} =& q^n + \Delta t \sum_{j = 1}^{i-1} \widetilde{a}_{ij} f(q^{(j)})
    + \Delta t \sum_{j = 1}^{i} {a_{ij}} g(q^{(j)}),\\
    q^{n + 1} =& q^n + \Delta t \sum_{i = 1}^{s} \widetilde{b}_i f(q^{(i)})
    + \Delta t \sum_{i = 1}^{s} b_i g(q^{(i)}).
  \end{aligned}
\end{equation}
The method associated with $(A,b)$ with $A = (a_{ij})$ is the implicit part, where the matrix $A$ is a lower triangular matrix.
On the other hand the method $(\widetilde{A}, \widetilde{b})$ with $\widetilde{A} = \widetilde{a}_{ij}$ is explicit.

In this work we consider two types of IMEX Runge-Kutta methods, denoted type I and type~II.
Type I methods have a combined Butcher tableau of the form
\begin{equation}\label{IMEX:type_I}
  \renewcommand\arraystretch{1.3}
  \begin{array}{c|ccc}
    \pmb{\widetilde{c}} & \widetilde{A} \\
    \hline
    & \pmb{\widetilde{b}^T}
  \end{array}
  \qquad
  \begin{array}{c|ccc}
    \pmb{c} & A \\
    \hline
    & \pmb{b^T}
  \end{array},
\end{equation}
where $A$ is invertible.
Type~II methods on the other hand are characterized by an explicit Euler step in the first stage resulting in a Butcher tableau of the form
\begin{equation}
  \renewcommand\arraystretch{1.3}
  \begin{array}{c|ccc}
    0 & 0 \\
    \pmb{\widehat{\widetilde{c}}} & \pmb{\widehat{\widetilde{\alpha}}} & \widehat{\widetilde{A}} \\
    \hline
    & \widehat{\widetilde{\beta}} & \pmb{\widehat{\widetilde{b}}}^T
  \end{array}
  \qquad
  \begin{array}{c|ccc}
    0 & 0 \\
    \pmb{c} & \widehat{\pmb{\alpha}} & \widehat{A} \\
    \hline
    & \widehat{\beta} & \widehat{\pmb{b}}^T
  \end{array},
\end{equation}
where $\widehat{A}$ is invertible.
If additionally $\widehat{\pmb{\alpha}} = 0$ and $\widehat\beta = 0$, the type~II method is said to have the ARS property.
For a more detailed discussion of both types see \cite{boscarino2024asymptotic} or \cite{biswas2025traveling}.
\subsection{Splitting of the BBMH system}
The BBMH system~\eqref{eq:BBMH} can be written as
\begin{equation}
 q_t + J q_x + \begin{pmatrix} 0 \\ -{w}/{\eps^2} \\ v \end{pmatrix} = 0,
\end{equation}
with $q = (u, v, w)^T$ and
\begin{equation}
 J = \begin{pmatrix}
          u & 1 & 0 \\ \frac 1 {\eps^2} & 0 & 0 \\ 0 & 0 & \eps^2
       \end{pmatrix}.
\end{equation}
The  matrix $J$ has the eigenvalues
\begin{equation}
  \lambda_1 = \eps^2, \quad
  \lambda_{2,3} = \frac{\eps^2u \pm \eps\sqrt{\eps^2u^2 + 4}}{2\eps^2},
\end{equation}
and eigenvectors
\begin{equation}
  \begin{aligned}
    &v_1 = (0,0,1)\\
    &v_2 = \left(\frac{1}{2}\left(u\eps^2 - \eps \sqrt{\eps^2u^2 + 4}\right), 1, 0\right)\\
    &v_3 = \left(\frac{1}{2}\left(u\eps^2 + \eps \sqrt{\eps^2u^2 + 4}\right), 1, 0\right).
  \end{aligned}
\end{equation}
In particular $J$ is diagonalizable for all values of $\eps > 0$, which has also been shown in~\cite{gavrilyuk2022hyperbolic}.
It seems intuitive to treat all linear terms implicitly, while treating the nonlinear terms explicitly, as is done in \cite{biswas2025traveling} for the hyperbolization of the KdV equation.
Therefore, the splitting
\begin{equation}\label{eq:first_splitting}
  q_t + \underbrace{\begin{pmatrix}  uu_x \\ 0 \\ 0 \end{pmatrix}}_{=:-f(q)} +\underbrace{\begin{pmatrix}  v_x \\ \frac{1}{\eps^2}u_x -\frac1 {\eps^2} w \\ \eps^2w_x + v \end{pmatrix}}_{=:-g(q)}  =  0
\end{equation}
seems natural.

To ensure that this splitting works well with IMEX methods, the systems obtained by replacing $f(q)$ or $g(q)$ by zero
must separately be hyperbolic, and the eigenvalues of $f'(q)$ have to be bounded (independently of $\eps$ as $\eps \to 0$) \cite{NoeSch2014}.
We say a splitting that satisfies these two properties is \emph{admissible}.
We will show in Theorem~\ref{thm:splitting_BBMH} that the splitting~\eqref{eq:first_splitting} is admissible.
Although these conditions are necessary for good numerical behavior, they are not sufficient:
it has been pointed out in~\cite{NoeSch2014} that, even for linear problems, not every admissible splitting will lead to a uniformly stable algorithm.
In fact, the splitting~\eqref{eq:first_splitting} is part of a three-parameter family of admissible splittings defined by
\begin{equation}
  \label{eq:splitting-family}
    q_t
    +
    \underbrace{\begin{pmatrix}
            u u_x + \delta_1 \eps v_x \\
      \frac{\delta_2}{\eps} u_x \\
      \delta_3 \eps^2 w_x
    \end{pmatrix}}_{=:-f(q)}
    +
    \underbrace{\begin{pmatrix}
            (1 - \delta_1 \eps) v_x \\
      \frac{1 - \delta_2 \eps}{\eps^2} u_x - \frac{1}{\eps^2} w \\
      (1 - \delta_3) \eps^2 w_x + v
    \end{pmatrix}}_{=:-g(q)}
     = 0,
\end{equation}
where $\delta_i \in [0,1]$ do not depend on $\eps$.
The splitting~\eqref{eq:first_splitting} is included by choosing $\delta_i = 0$.
\begin{theorem}\label{thm:splitting_BBMH}
  The splitting family~\eqref{eq:splitting-family} is admissible if the following conditions are satisfied:
  \begin{itemize}
    \item for all $i \in \{1,2,3\}$ we have $\delta_i \in [0,1]$,
    \item $\delta_1 \eps \leq 1$,
    \item $\delta_2 \eps \leq 1$,
    \item if $\delta_1 = 0$ or $\delta_2 = 0$, then both of them are zero,
    \item if $\delta_1 \eps = 1$ or $\delta_2 \eps = 1$, then both of them are unity.
  \end{itemize}
\end{theorem}
\begin{proof}
  For the function $f(q)$, we get the flux
  \begin{equation}
    F(q) :=
    \begin{pmatrix}
      -\frac{u^2}{2} -\delta_1 \eps v\\
      -\frac{\delta_2}{\eps} u\\
      -\delta_3 \eps^2 w\\
    \end{pmatrix},
  \end{equation}
  which has the Jacobian
  \begin{equation}
    F^\prime(q) =
    \begin{pmatrix}
      -u  & -\delta_1\eps & 0\\
      -\frac{\delta_2}{\eps} & 0 & 0\\
      0 & 0 & -\delta_3\eps\\
    \end{pmatrix}
  \end{equation}
  with the eigenvalues
  \begin{equation}
    \lambda_1 = -\delta_3 \eps \qquad \lambda_{2,3} = \frac{u}{2} \pm \sqrt{\frac{u^2}{4} + \delta_1\delta_2}.
  \end{equation}
  Since $\delta_1, \delta_2 \geq 0$ are positive all eigenvalues are real for all parameter choices.
  Furthermore, the eigenvectors are given by
  \begin{equation}
    \begin{aligned}
      v_1 &= \left( \frac{\eps \sqrt{1 - \delta_1 \eps}}{\sqrt{1 - \delta_2 \eps}}, 1, 0\right)^T\\
      v_2 &= \left(-\frac{\eps \sqrt{1 - \delta_1 \eps}}{\sqrt{1 - \delta_2 \eps}}, 1, 0\right)^T\\
      v_3 &= (0,0,1)^T\\
    \end{aligned}
  \end{equation}
if $\delta_1, \delta_2 \neq 0$. Thus, for this case we note that the matrix $F^\prime(q)$ is diagonalizable over $\mathbb{R}$. In the case that $\delta_1 = \delta_2 = 0$ we observe
that the matrix $F^\prime(q)$ is diagonalizable and therefore for all cases the subsystem $f(q)$ is hyperbolic.
On the other hand, the flux induced by the function $g(q)$ is
\begin{equation}
  G(q) =
  \begin{pmatrix}
    (1 - \delta_1\eps)v\\
    \frac{1 - \delta_2\eps}{\eps^2} u\\
    (1 - \delta_3)\eps^2 w
  \end{pmatrix},
\end{equation}
which has the Jacobian
\begin{equation}
  G^\prime(q) =
  \begin{pmatrix}
    0 & (1- \delta_1\eps) & 0\\
    \frac{1-\delta_2 \eps}{\eps^2} & 0 & 0\\
    0 & 0 & (1-\delta_3)\eps^2\\
  \end{pmatrix}
\end{equation}
with eigenvalues
\begin{equation}
  \lambda_1 = (1-\delta_3)\eps^2 \qquad \lambda_{2,3} = \pm \sqrt{(1 - \delta_1\eps) \frac{(1-\delta_2\eps)}{\eps^2}}.
\end{equation}
Since $\delta_1 \eps \leq 1$ and $\delta_2 \eps \leq 1$ the eigenvalues are real.
For the edge case $\delta_1 \eps = 1$ and $\delta_2 \eps = 1$ we observe that the matrix is already in diagonal form.
For $\delta_1 \neq 1$ and $\delta_2 \neq 1$ we have the eigenvectors
\begin{equation}
  \begin{aligned}
    v_1 &=  \left(\frac{\eps \sqrt{1 - \delta_1 \eps}}{\sqrt{1 - \delta_2 \eps}}, 1, 0\right)^T\\
    v_2 &= -\left(\frac{\eps \sqrt{1 - \delta_1 \eps}}{\sqrt{1 - \delta_2 \eps}}, 1, 0\right)^T\\
    v_3 &= \left(0,0,1\right)^T,
  \end{aligned}
\end{equation}
hence the subsystem $g(q)$ is hyperbolic under the given constraints.
\end{proof}
Since Theorem~\ref{thm:splitting_BBMH} allows us some choice, it seems desirable to make the splitting as explicit as possible, i.e., to reduce numerical cost,
while at the same time guaranteeing stability in the asymptotic limit $\eps \to 0$.
For the numerical experiments in Section~\ref{sec:numerical_experiments} we therefore use the splitting
\begin{equation}\label{eq:BBMH_split_1}
  q_t
  + \underbrace{\begin{pmatrix}
    u u_x \\
    0 \\
    \eps^2 w_x
  \end{pmatrix}}_{=:-f(q)}
  + \underbrace{\begin{pmatrix}
    v_x \\
    (u_x - w) / \epsilon^2 \\
    v
  \end{pmatrix}}_{=:-g(q)}
   = 0.
\end{equation}
This splitting is \eqref{eq:splitting-family} with $\delta_1 = \delta_2 = 0$ and $\delta_3 = 1$.
This choice ensures that the $\eps^2 w_x$ term is handled explicitly, reducing the size of the linear system to be solved.
Moreover, it avoids terms of the form $1 / \eps$ which may increase the numerical errors for small $\eps$ (depending on the implementation).

\subsection{Asymptotic preserving IMEX schemes}\label{sec:AP_results}
We investigate under which circumstances IMEX methods of type I and II admit to the AP property when applied to the splitting \eqref{eq:splitting-family}.
Similar to \cite{biswas2025traveling} we define
\begin{definition}
  An IMEX method is said to be \textit{globally stiffly accurate} (GSA) if
  \begin{equation}
    \widetilde{a}_{si} = \widetilde{b}_{i} \quad a_{si} = b_{i},
  \end{equation}
  for all $i \in \{1, \ldots, s\}$.
  In particular the implicit submethod is \textit{stiffly accurate}, while the explicit submethod possesses the \textit{first-same-as-last} (FSAL) property.
\end{definition}

\begin{definition}
  The initial data for the BBMH system is said to be well-prepared if \cite{boscarino2024asymptotic}
  $u(x, 0) = \uBBM(x, 0)$ and $w(x,0) = \uBBM_x(x,0)$.
\end{definition}
\begin{remark}
  From the previous definition it might seem odd that we do not put any constraints on the initial value of $v$.
  It turns out (see Theorems~\ref{thm:typeIAP} and~\ref{thm:typeIIAP}) that we do not need to constrain $v(x,t=0)$ in order to ensure the AP property.
\end{remark}
Using this definition we can generalize our observations from the previous section for type~I methods by stating
\begin{theorem}\label{thm:typeIAP}
  Given the BBMH system in the splitting~\eqref{eq:splitting-family} initialized with well-prepared initial data
  and $v(x,0) = v_0(x;\eps)$, where $v_0(x;\eps)$ is a sufficiently smooth function which may depend on $\eps$ in an arbitrary way.
  Then every GSA IMEX method of type~I is AP, i.e.,
  \begin{equation}
      u^n \to \uBBM^n, \qquad
      v^n \to -D_t \uBBM^n_{x}, \qquad
      w^n \to \uBBM^n_{x},
  \end{equation}
  for $\eps \to 0$, where $D_t$ denotes the discrete derivative operator given in \eqref{eq:thmAPI_Dt_v}.
\end{theorem}
\begin{proof}
  Let $\pmb{u} = [u^{(1)}, u^{(2)}, \dots, u^{(s)}]^T$, $\pmb{v} = [v^{(1)}, v^{(2)}, \dots, v^{(s)}]^T$, and $\pmb{w} = [w^{(1)}, w^{(2)}, \dots, w^{(s)}]^T$ denote the vectors of
  the stage-solution components corresponding to the variables $u$, $v$, and $w$, respectively. Furthermore, let $\pmb{1} = [1, 1, \dots, 1]^T$ represent the vector of ones in $\mathbb{R}^s$.
Also, let $q^n = [u^n, v^n, w^n]^T$ denote the numerical solution at $t^n$, and let $q^{(i)} = [u^{(i)},v^{(i)},w^{(i)}]^T$ denote the $i$-th stage for $i \in \{1, 2, \dots, s\}$. Then the
update of the solution by an IMEX method of type I \eqref{IMEX:type_I} applied to the splitting~\eqref{eq:splitting-family} in component form can be written as
\begin{subequations}
\begin{align}
    \pmb{u} &= u^n \pmb{1} - \Delta t  \tIAex (\pmb{u} \pmb{u}_x) - \Delta t \delta_1 \eps \tIAex \pmb{v}_x - \Delta t \tIAim(1-\delta_1 \eps) \pmb{v}_x \;, \label{Eq:stage_u_type_I} \\
    \pmb{v} &= v^n \pmb{1} - \Delta t \frac{\delta_2}{\eps} \tIAex \pmb{u}_x - \Delta t \left[ \frac{1-\delta_2 \eps}{\eps^2} \tIAim \pmb{u}_x - \frac{1}{\eps^2} \tIAim \pmb{w}\right] \;, \label{Eq:stage_v_type_I} \\
    \pmb{w} &= w^n \pmb{1} - \Delta t \delta_3 \eps^2 \tIAex \pmb{w}_x - \Delta t \left[(1-\delta_3)\eps^2 \tIAim \pmb{w}_x + \tIAim \pmb{v}\right] \;. \label{Eq:stage_w_type_I}
\end{align}
\end{subequations}
Similarly, the final update of the solution can be written in components as
\begin{subequations}
\begin{align}
  u^{n+1} & = u^n - \Delta t  \tIbex^T  (\pmb{u} \pmb{u}_x) - \Delta t \delta_1 \eps \tIbex^T \pmb{v}_x - \Delta t \tIbim^T (1-\delta_1 \eps) \pmb{v}_x  \;, \label{Eq:sol_u_type_I} \\
  v^{n+1} & = v^n - \Delta t \frac{\delta_2}{\eps} \tIbex^T  \pmb{u}_x - \Delta t \left[ \frac{1-\delta_2 \eps}{\eps^2} \tIbim^T \pmb{u}_x - \frac{1}{\eps^2} \tIbim^T  \pmb{w}\right] \;, \label{Eq:sol_v_type_I} \\
  w^{n+1} & = w^n - \Delta t \delta_3 \eps^2 \tIbex^T  \pmb{w}_x - \Delta t \left[(1-\delta_3)\eps^2 \tIbim^T  \pmb{w}_x + \tIbim^T \pmb{v}\right] \;. \label{Eq:sol_w_type_I}
\end{align}
\end{subequations}
Using the Hilbert expansion for the solution $q^n$ and the stages $q^{(i)}$ in \eqref{Eq:stage_v_type_I}, we obtain the leading order term as
\begin{equation}
  \tIAim ((\pmb{u}_0)_x - \pmb{w}_0) = \pmb{0}.
\end{equation}
Since $\tIAim$ is invertible, it has only a trivial kernel, and therefore $(\pmb{u}_0)_x  - \pmb{w}_0 = \pmb{0}$, or equivalently, $\pmb{w}_0 = (\pmb{u}_0)_x$.
By using the well-prepared initial data $w^n_0 = (u^n_{0})_x$ and $\pmb{w}_0 = (\pmb{u}_0)_x $ in the leading-order term of the expansion
of \eqref{Eq:stage_w_type_I}, we arrive at
\begin{equation}\label{Eq:stage_v_0_type_I}
  \pmb{v}_0 = -\tIAim^{-1}\left(\frac{\pmb{u}_0 - u_0^n \pmb{1}}{\Delta t}\right)_{x}.
\end{equation}
In the limit $\eps \to 0$, the stage equation \eqref{Eq:stage_u_type_I} for the $u$-component becomes
\begin{equation}\label{Eq:stage_u_0_type_I}
  \pmb{u}_0 = u_0^n \pmb{1} - \Delta t \tIAex \pmb{u}_0 (\pmb{u}_0)_x - \Delta t \tIAim (\pmb{v}_0)_x .
\end{equation}
Substituting $\pmb{v}_0$ from the previous equation, we obtain
\begin{equation}\label{Eq:stage_u_0_1_type_I}
  (\I - \partial_x^2) \left( \frac{\pmb{u}_0 - u_0^n \pmb{1}}{\Delta t} \right) = -\tIAex \pmb{u}_0 (\pmb{u}_0)_x.
\end{equation}
From the expansion of the solution update equation \eqref{Eq:sol_u_type_I} we obtain the leading order term
\begin{equation}\label{Eq:sol_u_0_type_I}
u^{n+1}_0 = u^n_0 - \Delta t  \tIbex^T  \pmb{u}_0 (\pmb{u}_0)_x - \Delta t \tIbim^T (\pmb{v}_0)_x \;.
\end{equation}
Using the last component of the equation \eqref{Eq:stage_u_0_type_I} we obtain
\begin{equation}
u^{(s)}_0 = u^n_0 - \Delta t  \widetilde{A}_s  \pmb{u}_0 (\pmb{u}_0)_x - \Delta t A_s (\pmb{v}_0)_x \;,
\end{equation}
where $A_s$ and $\widetilde{A}_s$ denote the $s$-th row of the matrix $A$ and $\widetilde{A}$, respectively.
Since the method is assumed to be GSA, comparing the last equation with equation \eqref{Eq:sol_u_0_type_I}, we conclude that $u^{n+1}_0=u^{(s)}_0$.
Now use $u^{n+1}_0=u^{(s)}_0$ in the last component of \eqref{Eq:stage_u_0_1_type_I} to obtain the update for the $u$-component
\begin{equation}\label{Eq:sol_u_0_1_type_I}
  (\I - \partial_x^2) \left( \frac{u^{n+1}_0 - u_0^n}{\Delta t} \right) = -\tIbex^T \pmb{u}_0 (\pmb{u}_0)_x.
\end{equation}
 The equation \eqref{Eq:stage_u_0_1_type_I} together with \eqref{Eq:sol_u_0_1_type_I} is equivalent to the numerical scheme one gets when applying the explicit part of the
 IMEX method to the BBM equation. This proves the AP property of the $u$-component.

The leading order term in the expansion of the solution-update equation \eqref{Eq:sol_w_type_I} for the $w$-component yields
\begin{equation}\label{Eq:sol_w_0_type_I}
  w^{n+1}_0 = w^n_0 - \Delta t \tIbim^T \pmb{v}_0 \;.
\end{equation}
Using \eqref{Eq:stage_v_0_type_I},
$u^{n+1}_0 = u^{(s)}_0$, and the fact that $\tIbim^T A^{-1}=[0,0,\ldots,0,1]^T$, we obtain
\begin{equation}
  \begin{aligned}
  \tIbim^T \pmb{v}_0
  & = -\tIbim^T \tIAim^{-1}\left(\frac{\pmb{u}_0 - u_0^n \pmb{1}}{\Delta t}\right)_{x}
    = -\left(\frac{u^{(s)}_0 - u_0^n}{\Delta t}\right)_{x}
    = -\left(\frac{{u^{n+1}_0} - u_0^n}{\Delta t}\right)_{x} \;.
  \end{aligned}
\end{equation}
Using this and the well-prepared initial date in \eqref{Eq:sol_w_0_type_I}, we get $w^{n+1}_0 = (u^{n+1}_0)_x$, proving the AP property for the $w$-component.

Finally, using the GSA property of the IMEX method, from equations \eqref{Eq:stage_v_type_I} and \eqref{Eq:sol_v_type_I} we obtain that
$v^{n+1} = v^{(s)}$ and hence $v^{n+1}_0 = v^{(s)}_0$. On the other hand, the last component of \eqref{Eq:stage_v_0_type_I} gives
\[
v^{(s)}_0 = -\sum_{j = 1}^{s} \alpha_{sj} \left(\frac{u^{(j)}_0 - u^n_0}{\Delta t}\right)_x ,
\]
where $\alpha_{ij}$ are the entries of $A^{-1}$, i.e., $A^{-1} = (\alpha_{ij})$. Hence, we obtain
\begin{equation}\label{eq:thmAPI_Dt_v}
  v^{n+1}_0 = -\sum_{j = 1}^{s} \alpha_{sj} \frac{(u^{(j)}_0)_x - (u^n_0)_x}{\Delta t} =: -D_t  (u^{n+1}_0)_x \;.
\end{equation}
This completes the proof of the AP property for the $v$-component.
\end{proof}

Similarly, we obtain for type~II methods
\begin{theorem}\label{thm:typeIIAP}
  Given the BBMH system in the splitting~\eqref{eq:splitting-family} initialized with well-prepared initial data and $v(x, t) = v_0(x;\eps)$, where $v_0(x;\eps)$ is a sufficiently smooth function. Then every GSA IMEX method of type~II is AP with respect to $u$ and $w$, i.e.
  \begin{equation}
      u^n {\rightarrow } \uBBM^n\ \qquad
      w^n {\rightarrow } \uBBM_{x}^n,
  \end{equation}
  for $\eps \rightarrow 0$. Furthermore, if the type~II method has the ARS property the $v$ component is AP as well in the sense that
  \begin{equation}
    v^n {\rightarrow } -D_t \uBBM^n_{x},
  \end{equation}
  where $D_t$ is the discrete time derivative operator given in \eqref{eq:APII_def_aplimit_v}.
\end{theorem}
\begin{proof}
The proof is similar to the proof of Theorem~\ref{thm:typeIAP}. We set $\pmb{u} = [u^{(2)}, \ldots, u^{(s)}]$, $\pmb{v} = [v^{(2)}, \ldots, v^{(s)}]$, $\pmb{w} = [w^{(2)}, \ldots, w^{(s)}]$, while the rest of the notation is identical to the previous theorem.
The corresponding calculations are given without much comment.
From the splitting we obtain
\begin{equation}
  \begin{aligned}
    \pmb{u} &= u^n \pmb{1} - \Delta t \widehat{\widetilde{\pmb{\alpha}}} u^{(1)}u_x^{(1)} - \Delta t \widehat{\widetilde{A}} \pmb{u}\pmb{u}_x
    - \Delta t \delta_1 \eps \widehat{\widetilde{\pmb{\alpha}}} v^{(1)}_x - \Delta t \delta_1 \eps \widehat{\widetilde{A}} \pmb{v}_x - \Delta t (1 - \delta_1 \eps) \widehat{\pmb{\alpha}} v^{(1)}_x - \Delta t (1-\delta_1 \eps) \widehat{A} \pmb{v}_x,\\
    \pmb{v} &= v^n \pmb{1} - \Delta t \widehat{\widetilde{\pmb{\alpha}}} \frac{\delta_2}{\eps} u^{(1)}_x - \Delta t \frac{\delta_2}{\eps} \widetilde{\widehat{A}} \pmb{u}_x
    -\Delta t \frac{1 - \delta_2 \eps}{\eps^2} \widehat{\pmb{\alpha}} u^{(1)}_x -\Delta t \frac{1 - \delta_2 \eps}{\eps^2} \widehat{A} \pmb{u}_x
    +\Delta t \frac{1}{\eps^2} \widehat{\pmb{\alpha}} w^{(1)} + \Delta t \frac{1}{\eps^2} \widehat{A} \pmb{w},\\
    \pmb{w} &= w^n \pmb{1} - \Delta t \pmb{\widehat{\widetilde{\alpha}}} \delta_3 \eps^2 w^{(1)}_x - \Delta t \widehat{\widetilde{A}} \delta_3 \eps^2 \pmb{w}_x
    -\Delta t (1 - \delta_3) \widehat{\pmb{\alpha}} \eps^2 w^{(1)}_x - \Delta t(1 - \delta_3) \widehat{A} \eps^2 \pmb{w}_x - \Delta t \widehat{\pmb{\alpha}} v^{(1)} - \Delta t \widehat{A} \pmb{v},
  \end{aligned}
\end{equation}
where terms of the form $\pmb{a} \pmb{b}$ are evaluated pointwise. Using the Hilbert expansion we get in the leading order for the stages $\pmb{v}$
\begin{equation}
  \begin{aligned}
    \widehat{\pmb{\alpha}} (u^{(1)}_{0,x} - w^{(1)}_0) + \widehat{A} (\pmb{u}_{0,x} - \pmb{w_0}) = 0.
  \end{aligned}
\end{equation}
Since $u^{(1)} = u^n$ for type~II methods and the initial data $u^n$ is well-prepared we have $u^{(1)}_{0,x} = w^{(1)}_0$. The remaining equation $\widehat{A} (\pmb{u}_{0,x} - \pmb{w_0}) = 0$
yields $\pmb{u}_{0,x} = \pmb{w_0}$, since $\widehat{A}$ is invertible. This shows the AP property of the $w$ component.
From the third equation we obtain in the leading order
\begin{equation}
  \begin{aligned}
    \pmb{w}_0 = w^n \pmb{1} - \Delta t \widehat{\pmb{\alpha}} \pmb{v} - \Delta t \widehat{A}\pmb{v}
  \end{aligned}
\end{equation}
which can be rewritten to
\begin{equation}\label{eq:proofAPII_v}
  \widehat{\pmb{\alpha}} v^{(1)} + \widehat{A}\pmb{v} = -\frac{\pmb{w}_0 - w^n \pmb{1}}{\Delta t}.
\end{equation}
Finally, from the first equation we get asymptotically in the leading order
\begin{equation}
  \begin{aligned}
    \pmb{u}_0 = u^n \pmb{1} - \Delta t \widehat{\widetilde{\pmb{\alpha}}} u^{(1)}_0 u^{(1)}_{0,x} - \Delta t \widehat{\widetilde{A}} \pmb{u}_0\pmb{u}_{0,x}
    -\Delta t \widehat{\pmb{\alpha}} v^{(1)}_x - \Delta t \widehat{A} v_x.
  \end{aligned}
\end{equation}
Inserting the identity from equation \eqref{eq:proofAPII_v} yields
\begin{equation}
  \pmb{u}_0 = u^n \pmb{1} - \Delta t \widehat{\widetilde{\pmb{\alpha}}} \pmb{u}_0\pmb{u}_{0,x} - \Delta t \widehat{\widetilde{A}} \pmb{u}_0\pmb{u}_{0,x}
    - (\pmb{u}_0 - u_0^n \pmb{1})_x,
\end{equation}
which we can rewrite to
\begin{equation}
  (I - \partial_x^2)(\pmb{u}_0 - u^n \pmb{1}) =  - \Delta t \widehat{\widetilde{\pmb{\alpha}}} \pmb{u}_0\pmb{u}_{0,x} - \Delta t \widehat{\widetilde{A}} \pmb{u}_0\pmb{u}_{0,x}.
\end{equation}
Since the method is assumed to be GSA we can deduce the AP property of the $u$ component.
Finally, if we assume the method to have the ARS property $\widehat{\pmb{\alpha}} = 0$ we obtain from equation~\eqref{eq:proofAPII_v}
\begin{equation}
  \pmb{v} = -\widehat{A}^{-1}\frac{\pmb{w}_0 - w^n \pmb{1}}{\Delta t}.
\end{equation}
Since the method is assumed to be GSA we obtain
\begin{equation}\label{eq:APII_def_aplimit_v}
  v^{(s)} = v^{n + 1} = A^{-1} \frac{u_0^{n + 1} - u_0^n}{\Delta t} = D_t u^{n + 1}.
  \qedhere
\end{equation}
\end{proof}

\begin{remark}\label{rem:AP_GSA}
  In both Theorems~\ref{thm:typeIAP} and~\ref{thm:typeIIAP} we assumed the method to be GSA.
  The necessity of this condition can be seen as follows.
  If one only considers stiffly-accurate methods, one has
  to compute the numerical update separately as well. For example the numerical update $v^{n + 1}$ for stiffly-accurate type~II methods reads
  \begin{equation}
    \begin{aligned}
      v^{n + 1} &= v^n - \Delta t \left( \tilde\beta \left(\frac{1- \delta_2 \eps}{\eps^2} u_x^{(1)} - \frac{1}{\eps^2} w^{(1)}\right) + \sum_{i = 2}^{s} b_{i-1} \left(\frac{1- \delta_2 \eps}{\eps^2} u_x^{(i)} - \frac{1}{\eps^2} w^{(i)}\right)\right)\\
                &- \Delta t \left(\widehat{\widetilde{\beta}} \left(\frac{\delta_2}{\eps} u_x^{(1)}\right) + \sum_{i = 2}^{s} \widehat{\widetilde{b}}_{i-1} \frac{\delta_2}{\eps} u_x^{(i)}\right).
    \end{aligned}
  \end{equation}
  The dominating $\epsilon^{-2}$ terms are
  \begin{equation}
  \begin{aligned}
    0 = \tilde \beta \left(u_x^{(1)} -  w^{(1)}\right) + \sum_{i = 2}^{s} b_{i-1} \left( u_x^{(i)} -  w^{(i)}\right).
  \end{aligned}
  \end{equation}
  This equation is satisfied trivially, since the asymptotics are already obtained by the stages. In order to guarantee
  \begin{equation}
    w^{n+1}_0 = {\left(u^{n+1}_{0}\right)}_{x},
  \end{equation}
  we need the FSAL condition to be fulfilled in the explicit submethod as well.  Similar arguments yield
  the same result for type~I methods.
\end{remark}

\section{Splitting the BBMH system for IMEX methods: discrete in space}\label{sec:discrete_in_space}

To conserve the quantities $\IBBM$ and $\IBBMH$ (see equations~\eqref{eq:I_BBM} and~\eqref{eq:BBMH_total}, respectively)
after a discretization in space, we use SBP operators and split forms of the equations.

\subsection{Summation-by-parts operators}
The conservation of a quantity $\mathcal{I}$ by solutions of a PDE can be expressed as
\begin{equation}
  \frac{d}{dt} \mathcal{I}(u(x,t)) = 0,
\end{equation}
which in our case is an integral equation. To ensure the discrete conservation of these quantities we will use a discrete
analog of integration by parts leading to a concept referred to as summation by parts.
This also means that we can only use integration by parts to show that a quantity is conserved analytically in order for this approach
to work.
This might require rewriting the PDE in such a way that the terms cancel out exactly when using integration by parts on the analytic level.
Since we are only considering periodic boundary conditions it suffices to only consider periodic SBP operators. We follow a combined approach of~\cite{ranocha2021broad} and \cite{Mattson_2017_SBP}.
\begin{definition}
  A periodic first-derivative SBP operator consists of a grid $\boldsymbol{x}$, a consistent first-derivative operator $D_1$, and a symmetric
  and positive matrix $M$ such that
  \begin{equation}
    \label{eq:periodic_SBP}
      M D_1 + D_1^T M = 0.
    \end{equation}
\end{definition}
A special case of SBP operators are the upwind operators, which are defined in
\begin{definition}
  A periodic first-derivative upwind SBP operator consists of a grid $\boldsymbol{x}$, consistent first-derivative operators $D_{\pm}$,
  and a symmetric and positive-definite matrix $M$ such that
  \begin{equation}
    \label{eq:periodic_upwind_SBP}
      M D_+ + D_-^T M = 0.
    \end{equation}
  Furthermore, the matrix
  \begin{equation}
    \frac{1}{2}M(D_+ - D_-)
  \end{equation}
  is negative semidefinite.
\end{definition}

\subsection{The BBM equation}

We use a conservative semidiscretization of the BBM equation~\eqref{eq:BBM}
developed in \cite{ranocha2021broad} for the variant with an additional linear
term $+u_x$. Adapted to our setting and using upwind SBP operators $D_\pm$ and
the induced central operator $D_1 = (D_+ + D_-)/2$, we obtain the semidiscretization
\begin{equation}
  \label{eq:bbm_semidiscretization}
    \partial_t \vec{\uBBM} = -\frac{1}{3} (\I - D_+ D_-)^{-1} \left(
      \vec{\uBBM} D_1 \vec{\uBBM} + D_1 \vec{\uBBM}^2
    \right),
\end{equation}
which ensures the discrete conservation of the conserved quantities. We have
\begin{theorem}[Variant of Theorem~4.1 of~\cite{ranocha2021broad}]
  The semidiscretization~\eqref{eq:bbm_semidiscretization} of
  the BBM equation~\eqref{eq:BBM} conserves the invariants
  \begin{equation}
    \vec{1}^T M \vec{u} \approx \int u,
    \qquad
    \frac{1}{2} \vec{u}^T M (\I - D_+ D_-) \vec{u} \approx \frac{1}{2} \int (u^2 + u_x^2),
  \end{equation}
  if $D_\pm$ are periodic first-derivative upwind SBP operators
  with diagonal mass matrix $M$ and $D_1 = (D_+ + D_-)/2$.
\end{theorem}
\begin{proof}
  Use similar arguments as in the proof of Theorem~4.1 in~\cite{ranocha2021broad}.
\end{proof}

\subsection{The BBMH system}

To discretize the splitting~\eqref{eq:splitting-family} of the
BBMH system in space we use periodic upwind SBP operators $D_\pm$ as
\begin{equation}
\label{eq:bbmh_semidiscretization}
  \partial_t
  \underbrace{\begin{pmatrix}
    \vec{u} \\
    \vec{v} \\
    \vec{w}
  \end{pmatrix}}_{= \vec{q}}
  =
  \underbrace{\begin{pmatrix}
  - \frac{1}{3} \left(\vec{u} D_1 \vec{u} + D_1 \vec{u}^2\right) - \delta_1 \eps D_+ \vec{v}\\
  -\frac{\delta_2}{\eps} D_- \vec{u}\\
  -\delta_3 \eps^2 D_1 \vec{w}\\
  \end{pmatrix}}_{= f(\vec{q})}
  + \underbrace{\begin{pmatrix}
      -(1 - \delta_1 \eps) D_+ \vec{v}\\
      -\frac{1-\delta_2 \eps}{\eps^2} D_- \vec{u} + \frac{1}{\eps^2} \vec{w}\\
      -(1- \delta_3)\eps^2 D_1 \vec{w} - \vec{v}
  \end{pmatrix}}_{= g(\vec{q})}.
\end{equation}

\begin{theorem}\label{thm:SD_BBMH}
  The semidiscretization~\eqref{eq:bbmh_semidiscretization} of the
  BBMH system~\eqref{eq:BBMH} conserves the invariants
  \begin{equation}
    \vec{1}^T M \vec{u} \approx \int u,
    \qquad
    \frac{1}{2} \left(
      \vec{u}^T M \vec{u}
      + \epsilon^2 \vec{v}^T M \vec{v}
      + \vec{w}^T M \vec{w}
    \right) \approx \frac{1}{2} \int (u^2 + \epsilon^2 v^2 + w^2),
  \end{equation}
  if $D_\pm$ are periodic first-derivative SBP operators with
  diagonal mass matrix $M$ and induced central operator
  $D_1 = (D_+ + D_-) / 2$.
\end{theorem}
\begin{proof}
  We have
  \begin{equation}
  \begin{aligned}
    \vec{1}^T M \partial_t \vec{u}
    &=
    - \vec{1}^T (1 - \delta_1\epsilon) M D_+ \vec{v}
    - \frac{1}{3} \vec{1}^T M (\vec{u} D_1 \vec{u} + D_1 \vec{u}^2)
    - \delta_1\epsilon \vec{1}^T M D_+ \vec{v}\\
    &=
    + (1 - \delta_1\epsilon) \vec{1}^T D_+^T M \vec{v}
    - \frac{1}{3} \left[\vec{u}^T M D_1 \vec{u} + \vec{1}^T M D_1 \vec{u}^2\right]
    + \delta_1\epsilon \vec{1}^T D_+^T M\vec{v}\\
    &=
    + (1 - \delta_1\epsilon)  \bigl(\underbrace{D_+ \vec{1} }_{= \vec{0}}\bigr)^T M \vec{v}
    - \frac{1}{3} \left[-\left(D_1\vec{u}\right)^T M \vec{u} - \left(\underbrace{D_1\vec{1}}_{ = 0}\right)^T M \vec{u}^2\right]
    + \delta_1\epsilon \bigl(\underbrace{D_+ \vec{1}}_{= \vec{0}}\bigr)^T M\vec{v}\\
    &=
	    - \frac{1}{3} \left[-\vec{u}^T M D_1\vec{u} \right]
	  =
		0,
  \end{aligned}
  \end{equation}
  where we have used the SBP property and the symmetry of the mass matrix $M$.
To show the invariance of the second quantity we observe
\begin{equation}
  \begin{aligned}
    \vec{u}^T M \partial_t \vec{u} &= \vec{u}^T M \left[-\frac{1}{3} (\vec{u} D_1 \vec{u} + D_1 \vec{u}^2) - \delta_1 \eps D_+ \vec{v}\right]
                                      + \vec{u}^T M \left[-(1 - \delta_1 \eps) D_+ v\right]\\
                                   &= \underbrace{-\frac{1}{3} (\vec{u}^2)^T MD_1 \vec{u} - \frac{1}{3} \vec{u}^T M D_1 \vec{u}^2}_{= 0} - \delta_1 \eps \vec{u}^T M D_+ v - (1 - \delta_1 \eps) \vec{u}^T M D_+ \vec{v}\\
                                   &=  -\vec{u}^T M D_+ \vec{v},
  \end{aligned}
\end{equation}
where we have used the SBP property. Furthermore, we have
\begin{equation}
  \begin{aligned}
    \eps^2 \vec{v}^T M \partial_t \vec{v} &= \eps^2 \vec{v}^T M [-\frac{\delta_2}{\eps} D_- \vec{u}] + \eps^2 \vec{v}^T M \left[- \frac{1 - \delta_2 \eps}{\eps^2} D_- \vec{u} + \frac{1}{\eps^2} \vec{w}\right]\\
                                          &= -\vec{v}^T M D_- \vec{u} + \vec{v}^T M \vec{w}.
  \end{aligned}
\end{equation}
Finally, we compute
\begin{equation}
  \begin{aligned}
    \vec{w}^T M \partial_t \vec{w} &= \vec{w}^T M [-\delta_3 \eps^2 D_1 \vec{w}] + \vec{w}^T M [-(1 - \delta_3)\eps^2 D_1 \vec{w} - \vec{v}]\\
                                   &= -\delta_3 \eps^2 \vec{w}^T M D_1 \vec{w} - \eps^2 \vec{w}^T M D_1 \vec{w} + \delta_3 \eps^2 \vec{w}^T M D_1 \vec{w} - \vec{w}^T M \vec{v}\\
                                   &= - \eps^2 \vec{w}^T M D_1 \vec{w} - \vec{w}^T M \vec{v}.
  \end{aligned}
\end{equation}
Combining these three observations leads to
\begin{equation}
  \begin{aligned}
    &\quad
    \vec{u}^T M \partial_t \vec{u}
    + \epsilon^2 \vec{v}^T M \partial_t \vec{v}
    + \vec{w}^T M \partial_t \vec{w}\\
   &= -\vec{u}^T M D_+ \vec{v} - \vec{v}^T M D_- \vec{u} + \vec{v}^T M \vec{w} - \eps^2 \vec{w}^T M D_1 \vec{w} - \vec{w}^T M \vec{v}
   = 0,
  \end{aligned}
\end{equation}
where we have used the SBP property.
\end{proof}
Since we want to study the AP property of the BBMH system discretely we have to ensure that we recover a sensible discretization for the limit
$\epsilon \rightarrow 0$. We study this behavior for both type~I and type~II methods and state
\begin{theorem}
  Given the BBMH semidiscretization~\eqref{eq:bbmh_semidiscretization} initialized with
  $u(x,0) = \eta(x,0)$, $w(x,0) = -\eta_x(x,0)$ and $v(x,0) = v_0(x;\eps)$, where $v_0$ is a sufficiently smooth function which may depend on $\eps$ in an arbitrary way,
  every GSA IMEX method of type I and II is AP for the components $u$ and $w$,
  i.e.
  \begin{equation}
    \vec{u}^n \rightarrow \vec{\eta}^n, \qquad
    \vec{w}^n \rightarrow D_- \vec{\eta}^n.
  \end{equation}
  Furthermore, GSA IMEX type I methods and GSA IMEX method of type~II with the ARS property are AP in the $v$ component, i.e.
  \begin{equation}
    \vec{v}^n \rightarrow -D_t D_1 \vec{\eta}^n,
  \end{equation}
  where the discrete derivative $D_t$ corresponds to equation~\eqref{eq:thmAPI_Dt_v} and equation~\eqref{eq:APII_def_aplimit_v} respectively.
\end{theorem}
\begin{proof}
  Use similar arguments as in the proofs of Theorems~\ref{thm:typeIAP} and~\ref{thm:typeIIAP}.
\end{proof}

\section{Numerical experiments}\label{sec:numerical_experiments}
In this section we conduct numerical tests that confirm the conservation and AP properties
of the proposed methods.  Additionally, we show experimentally that these methods exhibit
greatly improved error long-time growth compared to non-conservative methods.

We used Python with the libraries \texttt{numpy} \cite{harris2020array}, \texttt{scipy} \cite{virtanen2020scipy}, and \texttt{matplotlib} \cite{hunter2007matplotlib} for the traveling wave analysis, e.g., Figure~\ref{fig:peakon}.
The remaining numerical tests were performed using Julia \cite{bezanson2017julia} and SummationByPartsOperators.jl \cite{ranocha2021sbp}. For sparse linear systems, we use
UMFPACK from SuiteSparse \cite{davis2004umfpack,amestoy2004amd,davis2004colamd}
available from the Julia standard library.
All code and data required to reproduce the numerical results are available in our reproducibility repository \cite{bleecke2025asymptoticRepro}.

\subsection{Asymptotic preserving properties}\label{sec:num_test_AP}
To illustrate the results of Theorems~\ref{thm:typeIAP} and \ref{thm:typeIIAP},
we apply various IMEX methods to the semidiscretization~\eqref{eq:bbmh_semidiscretization} of the BBMH system.
For $u$ and $w$ we use well-prepared initial data as required by theorems~\ref{thm:typeIAP} and \ref{thm:typeIIAP}.
For $v$, we experiment by trying first the well-prepared initial condition
$v(x,0) = c D_1^2 \uBBM(x,0)$ (where $D_1$ is the discrete central difference
operator) and then the simple non-well-prepared choice $v(x,0)=0$.
We conduct our numerical tests with the following Runge-Kutta IMEX methods
\begin{itemize}
  \item \texttt{AGSA342}: A 2nd order GSA IMEX method of type I, see~\cite{boscarino2013}
  \item \texttt{SPIMEX322}: A 2nd order IMEX method of type I. It admits to a SA but not a GSA structure, see~\cite{Pareschi2005}
  \item \texttt{ARS443}: A 3rd order GSA IMEX method of type~II, see~\cite{ASCHER1997151}
  \item \texttt{BPR343}: A 3rd order GSA IMEX scheme of type~II that does not fulfill the ARS property, see~\cite{boscarino2013IMEX}
\end{itemize}
For the corresponding Butcher tableaux of these IMEX methods refer to the appendix of~\cite{boscarino2024asymptotic}.
We compute the $L^2$-norm of the error for each component $u,v,w$ of the BBMH system compared to the solution and
derivatives of the BBM equation ($\uBBM, -\uBBM_{tx}, \uBBM_x$).
For the parameters of the test we have used a finite difference approximation with $N = 2^9$ grid points using a time step of $\Delta t = 0.01$ and solve on the time interval
$[0,19.5]$.

The type I method \texttt{AGSA342} and type~II method \texttt{ARS443} fulfill all conditions of theorems~\ref{thm:typeIAP} and~\ref{thm:typeIIAP} respectively.
Since in both theorems we stated that the initial condition of $v$ does not need to be well-prepared we additionally test $v = 0$ as an initial condition.
The results for \texttt{AGSA342} are shown in Tables~\ref{tab:APtest_AGSA342} and~\ref{tab:APtest_AGSA342v0} respectively. The corresponding results for the \texttt{ARS443} method
are shown in Tables~\ref{tab:APtest_ARS443} and~\ref{tab:APtest_ARS443v0} respectively.
We observe that the error decreases in all four considered cases with a numerical convergence rate of around $1$, in essence showing that in this case all components are AP.
The results for the type I method \texttt{SPIMEX322} are shown in Table~\ref{tab:APtest_SPIMEX332}. Since this method does not admit to the GSA structure, which was one of the assumptions of theorem~\ref{thm:typeIAP},
we observe no convergence towards the numerical BBM solution.
Finally, the type~II method \texttt{BPR343} does admit to the GSA property but does not admit to the ARS property.
Therefore, it converges in the components $u$ and $w$ but not in the $v$ component, see Table~\ref{tab:APtest_BPR343}.

\begin{table}[htbp]
  \centering
  \caption{Error of the IMEX type I \texttt{AGSA342} (GSA) method with $N = 2^9$ grid points and $\Delta t = 0.01$ with initial condition $v(x,0) = c D_1^2 \uBBM$}
\begin{tabular}{ccccccc}
  \hline
  \textbf{$\varepsilon^2$} & \textbf{L2 error u} & \textbf{L2 EOC u} & \textbf{L2 error v} & \textbf{L2 EOC v} & \textbf{L2 error w} & \textbf{L2 EOC w} \\\hline
  1.00e-02 & 3.82e-03 &   & 1.82e-03 &   & 1.84e-03 &  \\
  1.00e-04 & 4.19e-05 & 0.98 & 3.04e-05 & 0.89 & 1.93e-05 & 0.99 \\
  1.00e-06 & 3.55e-06 & 0.54 & 1.18e-06 & 0.70 & 1.78e-06 & 0.52 \\
  1.00e-08 & 1.02e-07 & 0.77 & 3.24e-08 & 0.78 & 5.01e-08 & 0.78 \\
  1.00e-10 & 1.04e-09 & 1.00 & 3.35e-10 & 0.99 & 5.10e-10 & 1.00 \\\hline
\end{tabular}
\label{tab:APtest_AGSA342}
\end{table}

\begin{table}[htbp]
  \centering
  \caption{Error of the IMEX type I \texttt{AGSA342} method with $N = 2^9$ grid points and $\Delta t = 0.01$ with initial condition $v(x,0) = 0$}
\begin{tabular}{ccccccc}
  \hline
  \textbf{$\varepsilon^2$} & \textbf{L2 error u} & \textbf{L2 EOC u} & \textbf{L2 error v} & \textbf{L2 EOC v} & \textbf{L2 error w} & \textbf{L2 EOC w} \\\hline
  1.00e-02 & 3.98e-03 &   & 2.60e-02 &  & 3.20e-03 &  \\
  1.00e-04 & 4.22e-05 & 0.99 & 3.05e-05 & 1.46 & 1.94e-05 & 1.11 \\
  1.00e-06 & 3.51e-06 & 0.54 & 1.18e-06 & 0.71 & 1.77e-06 & 0.52 \\
  1.00e-08 & 1.01e-07 & 0.77 & 3.23e-08 & 0.78 & 4.97e-08 & 0.78 \\
  1.00e-10 & 1.03e-09 & 1.00 & 3.34e-10 & 0.99 & 5.06e-10 & 1.00 \\\hline
\end{tabular}
\label{tab:APtest_AGSA342v0}
\end{table}

\begin{table}[htbp]
  \centering
    \caption{Error of the IMEX type~II method \texttt{ARS443} (GSA, ARS)  with $N = 2^9$ grid points and $\Delta t = 0.01$ with initial condition $v(x,0) = c D_1^2 \uBBM$}
\begin{tabular}{ccccccc}
  \hline
  \textbf{$\varepsilon^2$} & \textbf{L2 error u} & \textbf{L2 EOC u} & \textbf{L2 error v} & \textbf{L2 EOC v} & \textbf{L2 error w} & \textbf{L2 EOC w} \\\hline
  1.00e-02 & 3.71e-03 &   & 3.94e-03 &   & 1.83e-03 &  \\
  1.00e-04 & 3.79e-05 & 1.00 & 2.69e-04 & 0.58 & 1.84e-05 & 1.00 \\
  1.00e-06 & 2.64e-07 & 1.08 & 1.03e-04 & 0.21 & 1.95e-07 & 0.99 \\
  1.00e-08 & 2.64e-09 & 1.00 & 1.52e-06 & 0.92 & 2.11e-09 & 0.98 \\
  1.00e-10 & 2.89e-11 & 0.98 & 1.53e-08 & 1.00 & 2.93e-11 & 0.93 \\\hline
\end{tabular}
\label{tab:APtest_ARS443}
\end{table}

\begin{table}[htbp]
  \centering
    \caption{Error of the IMEX type~II \texttt{ARS443} (GSA, ARS) method with $N = 2^9$ grid points and $\Delta t = 0.01$ with initial condition $v(x,0) = 0$}
\begin{tabular}{ccccccc}
  \hline
  \textbf{$\varepsilon^2$} & \textbf{L2 error u} & \textbf{L2 EOC u} & \textbf{L2 error v} & \textbf{L2 EOC v} & \textbf{L2 error w} & \textbf{L2 EOC w} \\\hline
  1.00e-02 & 4.96e-03 &   & 8.46e-02 &   & 7.29e-03 &  \\
  1.00e-04 & 3.82e-05 & 1.06 & 2.69e-04 & 1.25 & 1.85e-05 & 1.30 \\
  1.00e-06 & 2.64e-07 & 1.08 & 1.03e-04 & 0.21 & 1.95e-07 & 0.99 \\
  1.00e-08 & 2.67e-09 & 1.00 & 1.52e-06 & 0.92 & 2.14e-09 & 0.98 \\
  1.00e-10 & 2.92e-11 & 0.98 & 1.53e-08 & 1.00 & 2.95e-11 & 0.93 \\\hline
\end{tabular}
\label{tab:APtest_ARS443v0}
\end{table}

\begin{table}[htbp]
  \centering
  \caption{Error of the IMEX type I \texttt{SPIMEX332} (non GSA) method with $N = 2^9$ grid points and $\Delta t = 0.01$ with initial condition $v(x,0) = c D_1^2 \uBBM$}
\begin{tabular}{ccccccc}
  \hline
  \textbf{$\varepsilon^2$} & \textbf{L2 error u} & \textbf{L2 EOC u} & \textbf{L2 error v} & \textbf{L2 EOC v} & \textbf{L2 error w} & \textbf{L2 EOC w} \\\hline
  1.00e-02 & 3.71e-03 &   & 3.91e-03 &   & 1.83e-03 &  \\
  1.00e-04 & 3.77e-05 & 1.00 & 1.68e-04 & 0.68 & 1.87e-05 & 0.99 \\
  1.00e-06 & 8.35e-07 & 0.83 & 3.18e-04 & -0.14 & 1.40e-06 & 0.56 \\
  1.00e-08 & 1.03e-06 & -0.05 & 8.29e-04 & -0.21 & 1.67e-06 & -0.04 \\
  1.00e-10 & 1.03e-06 & -0.00 & 8.40e-04 & -0.00 & 1.68e-06 & -0.00 \\\hline
\end{tabular}
  \label{tab:APtest_SPIMEX332}
\end{table}

\begin{table}[htbp]
  \centering
  \caption{Error of the IMEX type~II \texttt{BPR343} (GSA, non ARS) method with $N = 2^9$ grid points and $\Delta t = 0.01$ with initial condition $v(x,0) = c D_1^2 \uBBM$}
\begin{tabular}{ccccccc}
  \hline
  \textbf{$\varepsilon^2$} & \textbf{L2 error u} & \textbf{L2 EOC u} & \textbf{L2 error v} & \textbf{L2 EOC v} & \textbf{L2 error w} & \textbf{L2 EOC w} \\\hline
  1.00e-02 & 3.71e-03 & & 3.92e-02 &  & 1.82e-03 &  \\
  1.00e-04 & 3.67e-05 & 1.00 & 3.86e-02 & 0.00 & 1.78e-05 & 1.01 \\
  1.00e-06 & 7.08e-07 & 0.86 & 3.86e-02 & 0.00 & 3.96e-07 & 0.83 \\
  1.00e-08 & 7.91e-09 & 0.98 & 3.86e-02 & -0.00 & 4.45e-09 & 0.97 \\
  1.00e-10 & 8.00e-11 & 1.00 & 3.86e-02 & -0.00 & 4.89e-11 & 0.98 \\\hline
\end{tabular}
  \label{tab:APtest_BPR343}
\end{table}

\subsection{Error growth in time}
It has been shown on theoretical level by~\cite{duran1998numerical} and~\cite{alvarez2012preservation} that conserving the quantity $\IBBM$
leads to linear error growth in comparison to a baseline quadratic error growth.
This has been later verified numerically by~\cite{ranocha2021broad}, when using the concept of numerical time relaxation.
In this section we test if we get the same improved error growth for the BBMH system. Additionally we investigate if this improved error growth
holds asymptotically if $\eps \rightarrow 0$.
Similar to~\cite{ranocha2021broad}, we use relaxation in time in order to conserve the additional SBP discretized quantity $\IBBM$ and $\IBBMH$ respectively, see \cite{Ranocha_2021}.
We first test the improved error growth of the BBMH system itself using a numerical reference solution generated by the Petviashvili method, see Section~\ref{sec:petv_derivation}.
We test for the value of $\eps = 10^{-3}$. We use the \texttt{ARS443} method with a step size $\Delta t = 0.5$ on a grid with $N = 2^8$ points and a 4th order upwind scheme.
In total $7.14$ domain traversals were performed leading to a final time of $t_{end} = 1071$.
The results are shown in Figure~\ref{fig:error_growth_petv}.
We observe that the relaxation method, which conserves $\IBBMH$, leads to a linear error growth while the non-conservative baseline method leads to a quadratic error growth in time.
Finally, we test if we get the same error growth if we use the analytical solution of the BBM equation. If one applies relaxation to the semidiscretization~\eqref{eq:bbm_semidiscretization} of the BBM equation
it is well known that one gets improved error growth, see~\cite{araujo2001error,alvarez2012preservation}.
\begin{figure}[htbp]
  \centering
  \includegraphics[scale = 0.5]{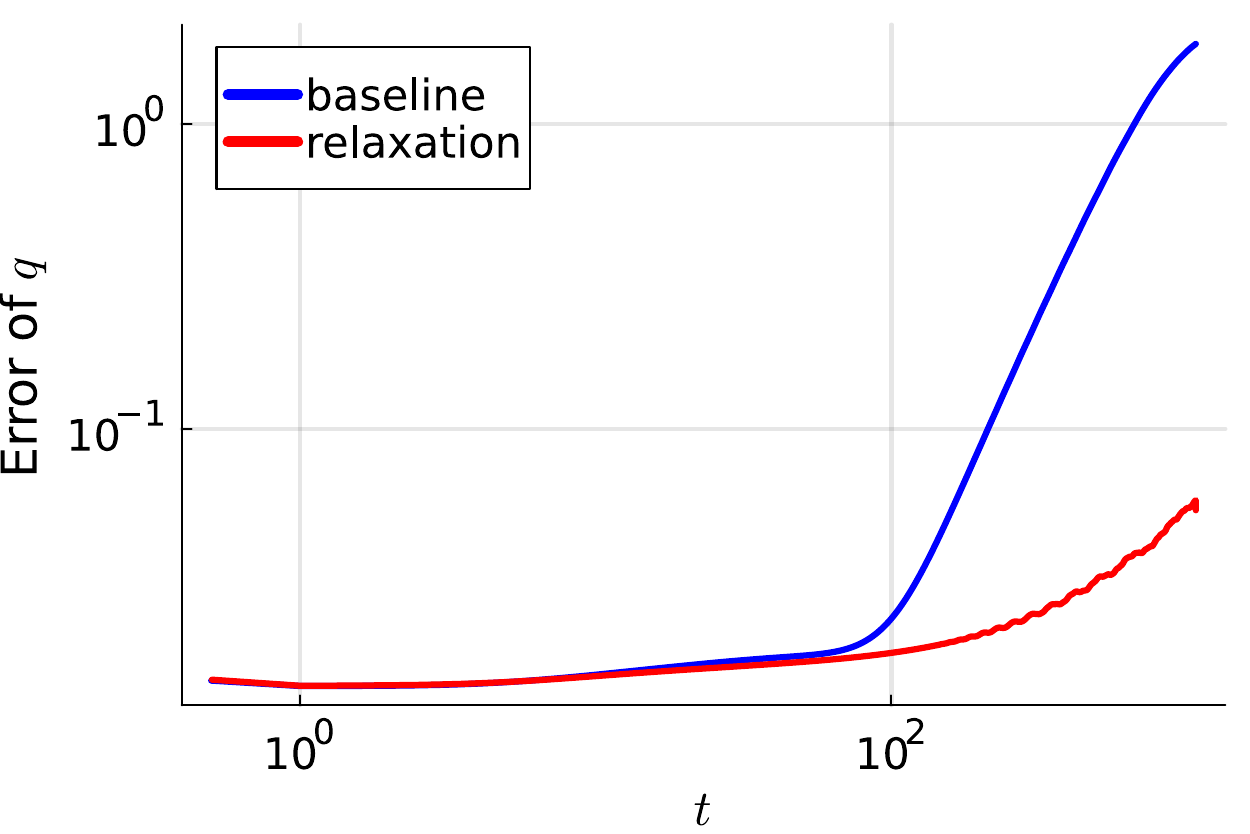}
  \caption{Error growth of the system variable $q$ with respect to the Petviashvili generated solution of the BBMH system for $\eps = 10^{-3}$.}
  \label{fig:error_growth_petv}
\end{figure}
We test the same phenomenon for the BBMH equation for different values of $\eps^2$. As a reference solution we use the analytical solution of the BBM equation given by
\begin{equation}\label{eq:ana_sol_bbm}
  u(x,t) = 1 + \frac{3(c - 1)}{\cosh(K(x - ct))^2},
  \qquad
  K = \frac{1}{2} \sqrt{\frac{c - 1}{c}},
\end{equation}
defined on the interval $[-90, +90]$ and propagated wave speed of $c = 1.2$.
The results are depicted in Figure~\ref{fig:error_growth}. We conduct tests with a range of values of $\eps$;
we show results for $\eps = 10^{-1}$ and $10^{-10}$, which are representative.
We observe that for the two values the error growth is improved when applying the relaxation method in time. For the value of $\eps = 10^{-1}$
we observe that the error of the BBMH is significantly higher than the BBM error growth with and without relaxation. For the second test $\eps = 10^{-10}$
there is no visible difference between the BBMH and BBM error.
\begin{figure}[htbp]
  \centering
  \begin{subfigure}[b]{0.44\textwidth}
    \includegraphics[width=\textwidth]{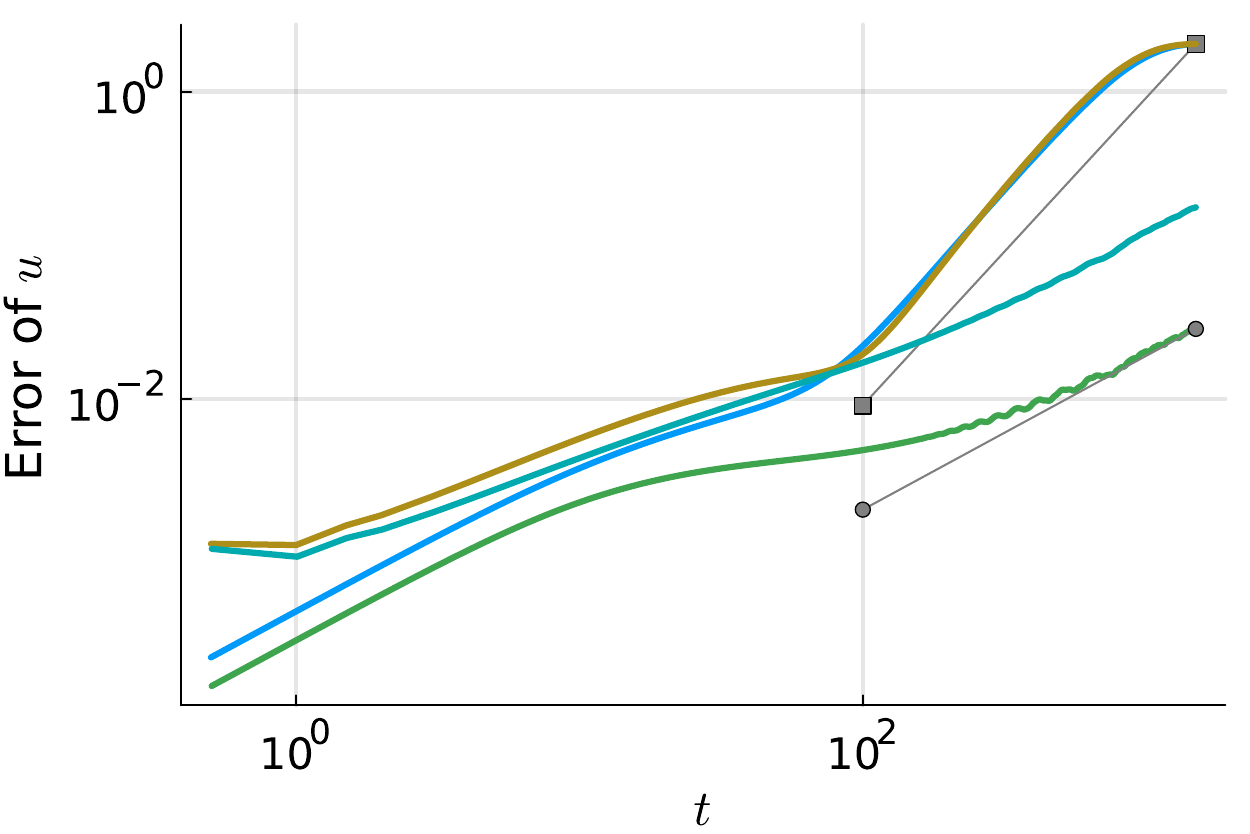}
    \caption{$\eps = 10^{-1}$}
    \label{fig:error_growth_eps1e-1}
  \end{subfigure}
  \hfill
  \begin{subfigure}[b]{0.44\textwidth}
    \includegraphics[width=\textwidth]{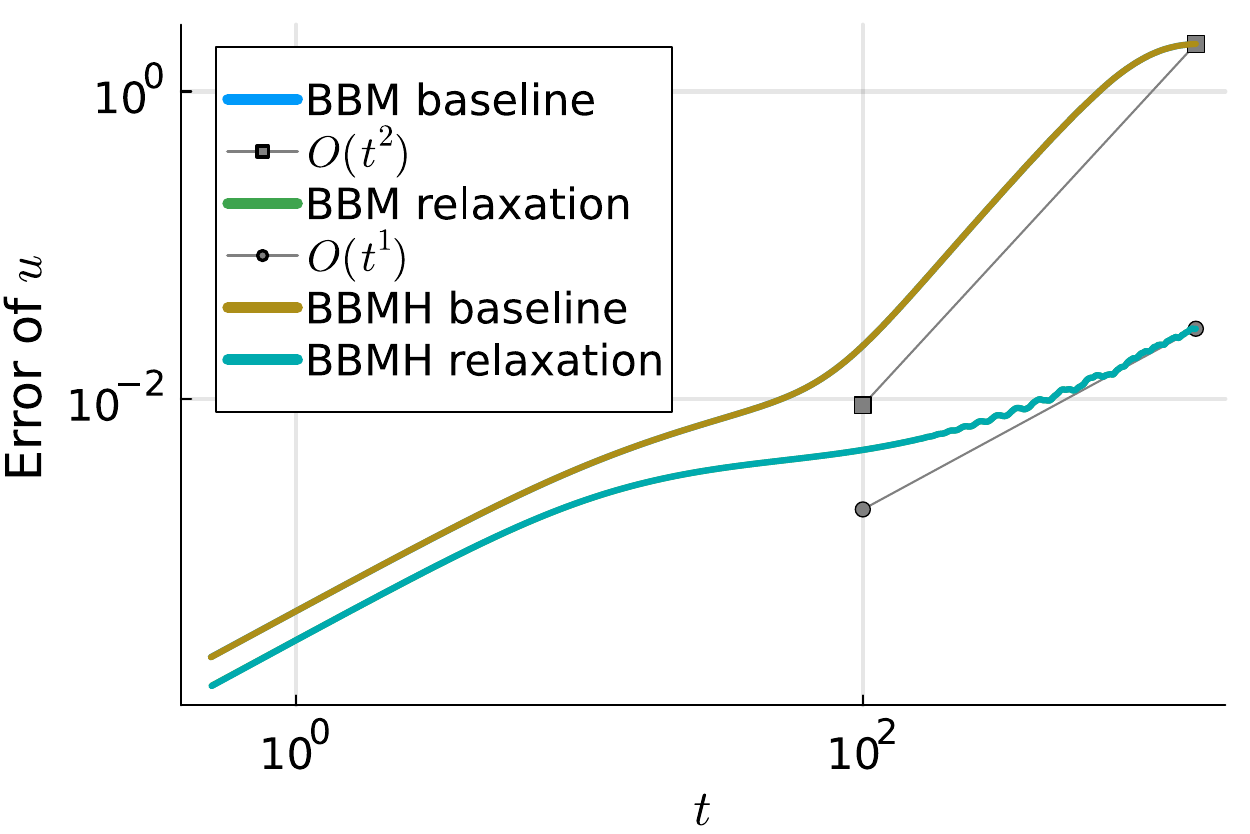}
    \caption{$\eps = 10^{-10}$}
    \label{fig:error_growth_1e-10}
  \end{subfigure}
  \caption{Error growth of the $u$ component with respect to the analytical solution of the BBM equation $\uBBM$.}
  \label{fig:error_growth}
\end{figure}

\section{Summary and conclusions}

In this work, we investigated a hyperbolic approximation of the BBM equation (BBMH), originally introduced by Gavrilyuk~\cite{gavrilyuk2022hyperbolic}.
This approximation offers a new way of solving nonlinear wave equations while using the numerical toolbox of hyperbolic PDEs.
Our numerical approach centered around IMEX time-stepping schemes of type I and II, specifically chosen for their GSA and asymptotic preserving (AP) properties in the limit $\eps \rightarrow 0$.
To ensure structural preservation of the conserved quantity of the BBM equation in discrete form, we applied SBP operators in space.
A reference solution for the BBMH system was computed using the Petviashvili algorithm,
and the numerical experiments validated the theoretical predictions—showing asymptotic consistency and improved long-time error behavior.

Comparing our results based on the hyperbolization of
Gavrilyuk and Shyue \cite{gavrilyuk2022hyperbolic} with
those obtained using the alternative hyperbolization mentioned
in Remark~\ref{rem:other_hyperbolic_approximation} (not discussed here)
indicates that the preservation of the Hamiltonian (relative
equilibrium) structure is crucial, not only to improve the long-time
error growth but also to achieve convergence to the BBM solution
as $\eps \to 0$; note that the convergence analysis in
\cite{giesselmann2025convergence} relies on the energy structure.
Thus, we will focus on developing structure-preserving hyperbolizations
and numerical methods for more complicated dispersive wave models in
future work, e.g., the BBM-BBM system.

\section*{Acknowledgments}

\input{funding}

\section*{Conflict of interest statement}

On behalf of all authors, the corresponding author states that there is no conflict of interest.

\printbibliography

\end{document}

%% file: macros.tex
\usepackage{xparse}

\let\epsilon\varepsilon
\let\phi\varphi
\let\rho\varrho

\usepackage{xspace}

\newcommand{\uBBM}{\eta}

\renewcommand{\vec}[1]{\underline{#1}}

\newcommand{\I}{\operatorname{I}}

\newcommand{\JBBM}{\mathcal{J}_\mathrm{BBM}}
\newcommand{\HBBM}{\mathcal{H}_\mathrm{BBM}}
\newcommand{\IBBM}{\mathcal{I}_\mathrm{BBM}}
\newcommand{\LBBM}{\mathcal{L}_\mathrm{BBM}}
\newcommand{\JBBMH}{\mathcal{J}_\mathrm{BBMH}}
\newcommand{\HBBMH}{\mathcal{H}_\mathrm{BBMH}}
\newcommand{\IBBMH}{\mathcal{I}_\mathrm{BBMH}}
\newcommand{\LBBMH}{\mathcal{L}_\mathrm{BBMH}}

\newcommand{\ut}{\tilde{u}}
\newcommand{\vt}{\tilde{v}}
\newcommand{\wt}{\tilde{w}}

\newcommand{\tIAex}{\tilde{A}}
\newcommand{\tIbex}{\pmb{\tilde{b}}}

\newcommand{\tIAim}{A}
\newcommand{\tIbim}{\pmb{b}}

%% file: abstract.tex
We study the hyperbolic approximation of the Benjamin-Bona-Mahony (BBM) equation
proposed recently by Gavrilyuk and Shyue (2022).
We develop asymptotic-preserving numerical methods using implicit-explicit
(additive) Runge-Kutta methods that are implicit in the stiff linear part.
The new discretization of the hyperbolization conserves important invariants
converging to invariants of the BBM equation.
We use the entropy relaxation approach to make the fully discrete schemes
energy-preserving.
Numerical experiments demonstrate the effectiveness of these discretizations.

%% file: funding.tex
SB and HR were supported by the Deutsche Forschungsgemeinschaft
(DFG, German Research Foundation, project numbers 513301895 and 528753982
as well as within the DFG priority program SPP~2410 with project number 526031774)
and the Daimler und Benz Stiftung (Daimler and Benz foundation,
project number 32-10/22).
DK was supported by funding from King Abdullah University of Science and Technology.
JS was supported by BOF funding of UHasselt.